\documentclass{article}
\usepackage{latexsym,amsthm,amssymb,amscd,amsmath}
\newtheorem{thm}{Theorem}[section]
\newtheorem{cor}[thm]{Corollary}
\newtheorem{prop}[thm]{Proposition}

\newtheorem{lem}[thm]{Lemma}

\newcommand{\be}{\begin{equation}}
\newcommand{\ee}{\end{equation}}
\newcommand{\ben}{\begin{enumerate}}
\newcommand{\een}{\end{enumerate}}
\newcommand{\beq}{\begin{eqnarray}}
\newcommand{\eeq}{\end{eqnarray}}
\newcommand{\beqn}{\begin{eqnarray*}}
\newcommand{\eeqn}{\end{eqnarray*}}

\newcommand{\pa}{\partial}

\newcommand{\g}{{\bf g}}

\newcommand{\pxi}{ {\pa \over \pa x^i}}

\title{Some Properties of  $m$-th Root Finsler Metrics}
\author{A. Tayebi, A. Nankali and E. Peyghan}
\begin{document}

\maketitle
\begin{abstract}
In this paper, we prove that every $m$-th  root metric with isotropic mean Berwald curvature reduces to a weakly Berwald metric. Then we show that an $m$-th  root metric with isotropic mean Landsberg curvature is a weakly Landsberg metric. We find necessary and sufficient condition under which  conformal $\beta$-change of an $m$-th  root metric be locally dually flat. Finally,  we prove that the conformal $\beta$-change of  locally projectively flat $m$-th root  metrics are locally Minkowskian.\\\\
{\bf {Keywords}}: Conformal change, $m$-th  root metric,  $\beta$-change, Locally dually flat metric, projectively flat metric.\footnote{ 2000 Mathematics subject Classification:
53C60, 53C25.}
\end{abstract}

\section{Introduction}
Let $(M, F)$ be a Finsler manifold of dimension $n$, $TM$ its tangent bundle and $(x^{i},y^{i})$ the
coordinates in a local chart on $TM$. Let $F$ be the following function on $%
M,$ by $F=\sqrt[m]{A}$, where $A$ is given by $A:=a_{i_{1}\dots i_{m}}(x)y^{i_{1}}y^{i_{2}}\dots y^{i_{m}}$ with $a_{i_{1}\dots i_{m}}$ symmetric in all its indices \cite{Mangalia}\cite{MatShim}\cite{Shim}\cite{TN1}\cite{TN2}. Then $F$ is called an $m$-th root Finsler metric. The theory of $m$-th root metric has been developed by Shimada \cite{Shim}, and applied to Biology as an ecological metric \cite{AIM}. It is regarded as a direct generalization of Riemannian metric in a sense, i.e., the second root metric is a Riemannian metric.

Let $(M, F)$ be a Finsler manifold of dimension $n$. Denote by $\tau(x,y)$ the
distortion of the Minkowski norm $F_{x}$ on $T_{x}M_{0}$, let
$\sigma(t)$ be the geodesic with $\sigma(0)=x$ and
$\dot{\sigma}(0)=y$. The rate of change of  $\tau(x,y)$ along Finslerian geodesics $\sigma(t)$ called $S$-curvature.
$F$ is said to have isotropic $S$-curvature and almost isotropic $S$-curvature  if ${\bf S}=(n+1)cF$ and ${\bf S}=(n+1)cF+dh$, respectively,  where $c=c(x)$ and $h=h(x)$ are  scalar functions on $M$ and $dh=h_{x^i}(x)y^i$ is the differential of $h$ \cite{TPSa}. Taking twice vertical covariant derivatives  of the $S$-curvature gives rise the $E$-curvature.  The Finsler metric $F$ is called weakly Berwald metric if ${\bf E}=0$ and is said to have isotropic mean Berwald curvature if ${\bf E}=\frac{n+1}{2}cF{\bf h}$, where $c=c(x)$ is a scalar function on $M$ and ${\bf h}=h_{ij}dx^idx^j$ is the angular metric.

\begin{thm}\label{mainthm3}
Let $F=\sqrt[m]{A}$ be an m-th root Finsler metric on an open subset
$U\subseteq \mathbb{R}^n$.\\
(i) For a scalar function $c=c(x)$  on $M$, the following are equivalent:
\begin{description}
\item[](ia)\ \  ${\bf S}=(n+1)cF+\eta$;
\item[](ib)\ \  ${\bf S}=\eta$.
\end{description}
(ii) For a scalar function $c=c(x)$  on $M$, the following are equivalent:
\begin{description}
\item[](iia)\ \ ${\bf E}=\frac{n+1}{2} c F{\bf h}$;
\item[](iib)\ \ ${\bf E}=0$.
\end{description}
\end{thm}

Let $(M, F)$ be a Finsler manifold. There are two basic tensors on Finsler manifolds:  fundamental metric tensor ${\bf g}_{y}$ and  the  Cartan torsion $\textbf{C}_y$, which are second and third order derivatives of ${1\over 2} F_x^2$ at  $y\in T_xM_0$, respectively. Taking a trace of Cartan torsion ${\bf C}_y$ give us the mean Cartan torsion ${\bf I}_y$.  The rate of change of the Cartan torsion along Finslerian geodesics, $\textbf{L}_y$ is said to be Landsberg curvature \cite{TP1}\cite{TP2}. Taking a trace of Landsberg curvature ${\bf L}_y$, yields the mean Landsberg curvature ${\bf J}_y$.  $F$ is called isotropic mean Landsberg curvature if ${\bf J}=cF{\bf I}$, where $c=c(x)$ is a scalar function on $M$.

\begin{thm}\label{mainthm4}
Let $(M,F)$ be an non-Riemannian $m$-th root Finsler
manifold. For a scalar function $c=c(x)$  on $M$, the following are equivalent:
\begin{description}
\item[](ia)\ \  ${\bf J}+cF{\bf I}=0$;
\item[](ib)\ \  ${\bf J}=0$.
\end{description}
\end{thm}

There are two important transformation in Finsler geometry: conformal change and $\beta$-change. Two metric functions $F$ and $\bar{F}$ on a manifold $M$ are called  conformal if the
length of an arbitrary vector in the one is proportional to the length in the
other, that is if $\bar{g_{ij}}=\varphi g_{ij}$. The length of vector $\varepsilon$ means here the fact that
$\varphi g_{ij}$, as well as $g_{ij}$, must be Finsler metric tensor. He showed that $\varphi$ falls into
a point function.  A change of Finsler metric $F\rightarrow \bar F$ is called a $\beta$-change of $F$, if $\bar F(x, y)=F(x, y)+\beta(x,y)$, where $\beta(x,y)=b_i(x)y^i$ is a $1$-form on a smooth manifold $M$. It is easy to see that,  if $\sup_{F(x,y)=1}|b_i(x)y^i|<1$, then $\bar F$ is again a Finsler metric.  The notion of a $\beta$-change has been proposed by Matsumoto, named by Hashiguchi-Ichijy\={o}  and studied in detail by Shibata \cite{HI}\cite{Mat}\cite{Shib}. If the Finsler metric $F$ reduces to a  Riemannian metric then ${\bar F}$ reduces to a Randers metric. Due to this reason, the $\beta$-change  has been called the Randers change of Finsler metric, also.

Let $(M,F)$ be a Finsler manifold. In this paper,  we are consider the conformal $\beta$-chenges of Finsler metrics
\[
\bar{F}=e^{\alpha(x)}F+\beta,
\]
where $\beta(x,y)=b_i(x)y^i$ is a $1$-form on a smooth manifold $M$ and $\alpha=\alpha(x)$ is  the conformal factor. It is easy to see that,  if $\sup_{F(x,y)=1}||\beta||<1$,  then $\bar F$ is again a Finsler metric.

Let  $F=\sqrt[m]{A}$ be an $m$-th root Finsler metric on an open
subset $U\subset \mathbb{R}^n$. Put
\[
A_{i}={\pa A\over \pa y^i}, \ \  A_{ij}={\pa^2 A\over \pa y^j\pa y^j}, \ \  A_{x^i}=\frac{\partial A}{\partial x^i}, \ \  A_0=A_{x^i}y^i.
\]
Suppose that $A_{ij}$ define a positive definite tensor and $A^{ij}$ denotes its inverse. The the following hold
\begin{eqnarray*}
&&g_{ij}=\frac{A^\frac{2}{m}-2}{m^2}[mAA_{ij}+(2-m)A_{i}A_{j}],\\
&&g^{ij}=A^{-\frac{2}{m}}[mAA^{ij}+\frac{m-2}{m-1}y^{i}y^{j}],\\
&&y^{i}A_{i}=mA,\ \ y^{i}A_{ij}=(m-1)A_{j},\ \
A^{ij}A_{i}=\frac{1}{m-1}y^{j},\\
&&y_{i}=\frac{1}{m}A^{\frac{2}{m}-1}A_{i},\
\ A_{i}A_{j}A^{ij}=\frac{m}{m-1}A.
\end{eqnarray*}

In \cite{am},  Amari-Nagaoka introduced the notion of dually flat
Riemannian metrics when they study the information geometry on
Riemannian manifolds.  A Finsler metric $F$ on an open subset
$U\subset \mathbb{R}^n$ is called dually flat if it satisfies
$(F^2)_{x^ky^l}y^k=2(F^2)_{x^l}$ \cite{shen}\cite{TPSa}.

In this paper, we consider conformal $\beta$-change of locally dually flat $m$-th
root Finsler metrics and prove the following.
\begin{thm}\label{mainthm1}
Let  $F=\sqrt[m]{A}$ be an $m$-th root Finsler metric on an open
subset $U\subset \mathbb{R}^n$, where $A$ is irreducible. Suppose
that $\bar{F}=e^\alpha F+\beta$ be conformal $\beta$-change of $F$
where $\beta=b_i(x)y^i$ and $\alpha=\alpha(x)$. Then $\bar{F}$ is
locally dually flat if and only if there exists a 1-form $\theta =
\theta_{l} (x)y^l$  on U such that the following hold
\begin{eqnarray}
&&\beta_{0l}\beta +\beta_l\beta_0=2\beta\beta _{x^{l}},\label{SRR1}\\
&&A_{x^l}=\frac{1}{3m}[mA\theta_{l}+2\theta
A_{l}+2(\alpha_{0}A_{l}-\alpha_{x^{l}}A)],\label{SRR2}\\
&&\beta[(\frac{1}{m}-2)A_lA^{-1}A_0 -4A_{x^{l}}+
\alpha_{0}A_{l}]+2[A_l\beta_0+(A_{0}\beta)_{l}] =-2me^\alpha A\Psi,\qquad
\end{eqnarray}
where $\beta _{0l}=\beta _{x^k
y^l}y^k$, $\alpha_{0}=\alpha_{x^{l}}y^{l}$,
$\beta_{x^l}=(b_i)_{x^l}y^i$, $\beta_0=\beta_{x^l}y^i$, $\beta_{0l}=(b_l)_0$ and  $\Psi=\alpha_0\beta_l+\beta_{0l}-2\beta_{x^{l}}-2\alpha_{x^l}\beta$.
\end{thm}

A Finsler metric is said to be  locally projectively flat if at any
point there is a local coordinate system  in which the geodesics are
straight lines as point sets. It is known that a Finsler metric
$F(x,y)$ on an open domain $ U\subset \mathbb{R}^n$   is  locally
projectively flat  if and only if $G^i= Py^i$, where $P(x,  \lambda
y) = \lambda P(x, y)$, $\lambda >0$ \cite{shLi1}. Finally, we study conformal $\beta$-change of locally projectively flat $m$-th root  metrics and prove the following.

\begin{thm}\label{mainthm2}
Let  $F=\sqrt[m]{A}$ be an $m$-th root Finsler metric on an open
subset $U\subset \mathbb{R}^n$, where $A$ is irreducible. Suppose
that $\bar{F}=e^\alpha F+\beta$ be conformal $\beta$-change of $F$
where $\beta=b_i(x)y^i$ and $\alpha=\alpha(x)$. Then  $\bar{F}$ is
locally projectively flat if and only if it is locally Minkowskian.
\end{thm}

\bigskip

\section{Preliminaries}

Let $M$ be a n-dimensional $ C^\infty$ manifold. Denote by $T_x M $ the tangent space at $x \in M$, by $TM=\cup _{x \in M} T_x M $ the tangent bundle of $M$ and by $TM_{0} = TM \setminus \{ 0 \}$ the slit tangent bundle. A  Finsler metric on $M$ is a function $ F:TM \rightarrow [0,\infty)$ which has the following properties:
(i) $F$ is $C^\infty$ on $TM_{0}$; (ii) $F$ is positively 1-homogeneous on the fibers of tangent bundle $TM$; (iii) for each $y\in T_xM$, the following quadratic form ${\bf g}_y$ on $T_xM$  is positive definite,
\[
{\bf g}_{y}(u,v):={1 \over 2} \frac{\partial^2}{\partial s \partial t}\left[  F^2 (y+su+tv)\right]|_{s,t=0}, \ \
u,v\in T_xM.
\]
Let  $x\in M$ and $F_x:=F|_{T_xM}$.  To measure the
non-Euclidean feature of $F_x$, define ${\bf C}_y:T_xM\otimes T_xM\otimes
T_xM\rightarrow \mathbb{R}$ by
\[
{\bf C}_{y}(u,v,w):={1 \over 2} \frac{d}{dt}\left[{\bf g}_{y+tw}(u,v)
\right]|_{t=0}, \ \ u,v,w\in T_xM.
\]
The family ${\bf C}:=\{{\bf C}_y\}_{y\in TM_0}$  is called the Cartan torsion. It is well known that {\bf{C}=0} if and
only if $F$ is Riemannian.

\bigskip

Given a Finsler manifold $(M,F)$, then a global vector field $G$ is
induced by $F$ on $TM_0$, which in a standard coordinate $(x^i,y^i)$
for $TM_0$ is given by ${\bf{G}}=y^i {{\partial} \over {\partial x^i}}-2G^i(x,y){{\partial} \over
{\partial y^i}}$, where $G^i(y)$ are local functions on $TM$. $\bf{G}$ is called the  associated spray to $(M,F)$. The projection of an integral curve of $\bf{G}$ is called a geodesic in $M$. In local coordinates, a curve $c(t)$ is a geodesic if and only if its coordinates $(c^i(t))$ satisfy $ \ddot c^i+2G^i(\dot c)=0$.

\bigskip

Define ${\bf B}_y:T_xM\otimes T_xM \otimes T_xM\rightarrow T_xM$ and ${\bf E}_y:T_xM \otimes T_xM\rightarrow \mathbb{R}$ by ${\bf B}_y(u, v, w):=B^i_{\ jkl}(y)u^jv^kw^l{{\partial } \over {\partial x^i}}|_x$, ${\bf E}_y(u,v):=E_{jk}(y)u^jv^k$, where
\[
B^i_{\ jkl}(y):={{\partial^3 G^i} \over {\partial y^j \partial y^k \partial y^l}}(y),\ \ \  E_{jk}(y):={{1}\over{2}}B^m_{\ jkm}(y),
\]
$u=u^i{{\partial } \over {\partial x^i}}|_x$, $v=v^i{{\partial } \over {\partial x^i}}|_x$ and $w=w^i{{\partial } \over {\partial x^i}}|_x$. $\bf B$ and $\bf E$ are called the Berwald curvature and mean Berwald curvature, respectively. A Finsler metric is called a Berwald metric and mean Berwald metric if $\textbf{B}=0$ or ${\bf E}=0$, respectively.

\bigskip

Let
\[ \tau (x, y):=\ln \Big [ { \sqrt{\det \Big ( g_{ij}(x,y) \Big )}\over {\rm Vol} ({\rm B}^n(1)) }
 \cdot {\rm Vol} \Big \{ (y^i) \in \mathbb{R}^n \Big | \ F\Big (y^i \pxi|_x \Big ) < 1 \Big \} \Big ] .\]
$\tau=\tau(x,y)$ is a scalar function on $TM\setminus\{0\}$, which is called  the  distortion.

Let
\[ {\bf S}(x, y):= { d \over dt} \Big [ \tau \Big ( \sigma(t), \dot{\sigma}(t) \Big ) \Big ]_{t=0},
\]
where $\sigma(t)$ is the geodesic with $\sigma(0)=x$ and $\dot{\sigma}(0)=y$.
${\bf S}$ is called the  S-curvature. ${\bf S}$ said to be  {\it isotropic} if there is a scalar functions
$c(x)$ on $M$ such that
\[ {\bf S}(x, y) = (n+1) c(x) F(x, y).\]

\section{ Proof of the Theorem \ref{mainthm3}}
In  local coordinates $(x^i,y^i)$, the vector filed ${\bf G}=y^i\frac{\pa}{\pa x^i}-2G^i\frac{\pa}{\pa y^i}$ is a global vector field on $TM_0$, where $G^i=G^i(x,y)$ are local functions on $TM_0$  given by following
\[
G^i:=\frac{1}{4}g^{il}\Big[ \frac{\partial^2F^2}{\partial x^k \partial y^l}y^k-\frac{\partial F^2}{\partial x^l}\Big],\ \ \ \ y\in T_xM.
\]
By a simple calculation, we have the following.
\begin{lem}\label{lemG}{\rm (\cite{YY})}
\emph{Let $F=\sqrt[m]{A}$  be an $m$-th root Finsler metric on an open subset
$U\subseteq\mathbb{R}^n$. Then the spray coefficients of $F$ are
given by:}
\[
G^i=\frac{1}{2}(A_{0j}-A_{x^j})A^{ij}.
\]
\end{lem}
Thus the spray coefficients of an  $m$-th root Finsler metric are rational functions with respect to $y$.

\bigskip

\begin{lem}\label{lem1}
Let $F=\sqrt[m]{A}$  be an $m$-th root Finsler metric on an open subset
$U\subseteq\mathbb{R}^n$. Then the following are equivalent
\begin{description}
\item[a)] ${\bf S}=(n+1)cF+\eta$;
\item[b)] ${\bf S}=\eta$;
\end{description}
where $c=c(x)$ is a scalar function and $\eta=\eta_i(x)y^i$ is a 1-form on $M$.
\end{lem}
\begin{proof}
By Lemma \ref{lemG}, the $E$-curvature of an $m$-th root  metric is a rational function in $y$. On the other hand, by taking twice vertical covariant derivatives  of the S-curvature, we get  the $E$-curvature. Thus $S$-curvature is a rational function in $y$. Suppose that $F$ has almost isotropic $S$-curvature, $\textbf{S}=(n+1)c(x)F+\eta$, where $c=c(x)$ is a scalar function and $\eta=\eta_i(x)y^i$ is a 1-form on $M$. Then the left hand side of $\textbf{S}-\eta=(n+1)c(x)F$ is a rational function in $y$ while the right hand is irrational function. Thus $c=0$ and ${\bf S}=\eta$.
\end{proof}

\begin{lem}\label{lem2}
Let $F=\sqrt[m]{A}$  be an $m$-th root Finsler metric on an open subset
$U\subseteq\mathbb{R}^n$. Then the following are equivalent
\begin{description}
\item[a)] ${\bf E}=\frac{n+1}{2}cF{\bf h}$;
\item[b)]${\bf E}=0$;
\end{description}
where $c=c(x)$ is a scalar function on $M$.
\end{lem}
\begin{proof}
Suppose that $F=\sqrt[m]{A}$ has isotopic mean Berwald curvature
\[
{\bf E}=\frac{n+1}{2}cF{\bf h},
\]
where $c=c(x)$ is a scalar function on $M$. The left hand side of ${\bf E}=\frac{n+1}{2}cF{\bf h}$ is a rational function in $y$ while the right hand is irrational function. Thus $c=0$ and ${\bf E}=0$.
\end{proof}

\bigskip

\noindent {\it\bf Proof of Theorem \ref{mainthm3}}: By Lemmas \ref{lem1} and \ref{lem2}, we get the proof.
\qed

\bigskip

By the  Theorem \ref{mainthm3}, we have the following:
\begin{cor}
Let $F=\sqrt[m]{A}$  be an $m$-th root Finsler metric on an open subset
$U\subseteq\mathbb{R}^n$. Suppose that $F$ has isotropic $S$-curvature ${\bf S}=(n+1)cF$, for some scalar function $c=c(x)$ on $M$. Then ${\bf S}=0$.
\end{cor}

\bigskip

A Finsler metric $F$ satisfying $F_{x^k}=FF_{y^k}$ is called a Funk metric. The standard Funk metric on the Euclidean unit ball $B^n(1)$ is denoted by $\Theta$ and defined by
\begin{equation}\label{Funk}
  \Theta(x,y):=\frac{\sqrt{|y|^2-(|x|^2|y|^2-<x,y>^2)}+<x,y>}{1-|x|^2}, \,\,\,\,\,y\in T_xB^n(1)\simeq \mathbb{R}^n,\nonumber
\end{equation}
where $<,>$ and $| . |$ denote the Euclidean inner product and norm on $\mathbb{R}^n$, respectively. In  \cite{CS},
Chen-Shen introduce the notion of isotropic Berwald metrics. A Finsler metric $F$ is said to be isotropic Berwald metric if
its Berwald curvature is in the following form
\begin{equation}\label{IBCurve}
B^i_{\ jkl}=c\{F_{y^jy^k}\delta^i_{\ l}+F_{y^ky^l}\delta^i_{\ j}+F_{y^ly^j}\delta^i_{\ k}+F_{y^jy^ky^l}y^i\},
\end{equation}
for some scalar function $c=c(x)$ on $M$. Berwald metrics are trivially isotropic Berwald metrics with $c=0$. Funk metrics are also non-trivial isotropic Berwald metrics. In  (\ref{IBCurve}), putting $i=l$ yields
\[
E_{ij}=\frac{n+1}{2}cF^{-1}h_{ij}.
\]
Plugging it in  (\ref{IBCurve}) implies that
\begin{equation}
B^i_{\ jkl}=\frac{2}{n+1}\{E_{jk}\delta^i_l+E_{kl}\delta^i_j+E_{lj}\delta^i_k+E_{jk,l}y^i\}.
\end{equation}
This means that every  isotropic Berwald metric is a Douglas metric. For the definition of Douglas metrics see \cite{BM}.

Now,  let $F=\sqrt[m]{A}$  be an $m$-th root Finsler metric on an open subset
$U\subseteq\mathbb{R}^n$. Suppose that $F$ has isotropic Berwald curvature (\ref{IBCurve}). By Lemma \ref{lemG}, the left hand side of (\ref{IBCurve}) is a rational function in $y$ while the right hand is irrational function. Thus $c=0$ and we have the following.

\begin{thm}\label{IBM}
Let $F=\sqrt[m]{A}$  be an $m$-th root Finsler metric on an open subset
$U\subseteq\mathbb{R}^n$. Suppose that $F$ has isotropic Berwald curvature. Then $F$ is a Berwald metric.
\end{thm}

In \cite{TR}, Tayebi-Rafie Rad proved that every isotropic Berwald metric (\ref{IBCurve}) on a manifold $M$ has isotopic $S$-curvature ${\bf S}=(n+1)cF$, for some scalar function $c=c(x)$ on $M$. Thus by the Theorem \ref{IBM}, we have the following.
\begin{cor}
Let $F=\sqrt[m]{A}$  be an $m$-th root Finsler metric on an open subset
$U\subseteq\mathbb{R}^n$. Suppose that $F$ has isotropic Berwald curvature. Then ${\bf S}=0$.
\end{cor}
\section{ Proof of the Theorem \ref{mainthm4}}
The quotient ${\bf J}/{\bf I}$ is regarded as the relative rate of change of mean Cartan torsion ${\bf I}$ along Finslerian geodesics.  Then $F$  is said to be isotropic mean Landsberg metric if ${\bf J}=cF \bf I$, where $c=c(x)$ is a scalar function on $M$. In this section, we are going to prove  the Theorem \ref{mainthm4}. More precisely, we show that every  $m$-th root isotropic mean Landsberg metric reduces to a weakly Landsberg metric.

\bigskip

\noindent {\it\bf Proof of Theorem \ref{mainthm4}}:
The mean Cartan tensor of $F$ is given by following
\begin{eqnarray*}
I_i&=&g^{jk}C_{ijk}\\
&=&\frac{1}{m}A^{-3}\big[mAA^{jk}+\frac{m-2}{m-1}y^jy^k\big]\\
&&\times\Big[A^{2}A_{ijk}+(\frac{2}{m}-1)\{(\frac{2}{m}-2)A_{i}A_{j}A_{k}+A[A_{i}A_{jk}+A_{j}A_{ki}+A_{k}A_{ij}]\}\Big].
\end{eqnarray*}
The mean Landsberg curvature of $F$ is given by
\begin{eqnarray*}
J_i&=&g^{jk}L_{ijk}\\
&=&A^{-\frac{2}{m}}\big[mAA^{jk}+\frac{m-2}{m-1}y^jy^k\big]\big[-\frac{1}{2m}A^{\frac{2}{m}-1}A_sG^s_{ijk}\big]\\
&=&-\frac{1}{2m}A^{-1}A_sG^s_{ijk}\big[mAA^{jk}+\frac{m-2}{m-1}y^jy^k\big].
\end{eqnarray*}
Since ${\bf J}=cF \bf I$, then we have
\[
A_sG^s_{ijk}=-2cA^{\frac{1}{m}-2}\Big[A^{2}A_{ijk}+(\frac{2}{m}-1)\{(\frac{2}{m}-2)A_{i}A_{j}A_{k}+A[A_{i}A_{jk}+A_{j}A_{ki}+A_{k}A_{ij}]\}\Big].
\]
By the Lemma \ref{lemG}, the left hand side is a rational function in
$y$, while its right-hand side is an irrational function in $y$.
Thus, either $c=0$ or $A$ satisfies the following PDE:
\[
A^{2}A_{ijk}+(\frac{2}{m}-1)(\frac{2}{m}-2)A_{i}A_{j}A_{k}+(\frac{2}{m}-1)A\{A_{i}A_{jk}+A_{j}A_{ki}+A_{k}A_{ij}\}=0.
\]
That implies that $C_{ijk}=0$. Hence, by Deike's theorem, $F$ is
Riemannian metric, which contradicts our assumption. Therefore,
$c=0$. This completes the proof.
\qed

\bigskip

By the similarly method, we have the following.
\begin{thm}
Let $F=\sqrt[m]{A}$ be an  non-Riemannian  m-th root Finsler metric on an open subset
$U\subseteq \mathbb{R}^n$. Suppose that $F$ has isotropic  Landsberg curvature, i.e., ${\bf L}=cF{\bf C}$;
where $c=c(x)$ is a scalar function on $M$. Then $F$ reduces to a Landsberg metric.
\end{thm}

\section{ Proof of the Theorem \ref{mainthm1}}
A Finsler metric $F=F(x,y)$ on a manifold $M$ is said to be locally dually flat if at any point there is a  coordinate system
$(x^i)$ in which the spray coefficients are in the
following form
\[
G^i = -\frac{1}{2}g^{ij}H_{y^j},
\]
where $H=H(x, y)$ is a $C^\infty$ scalar function on $TM_{0}=TM \setminus \{ 0 \}$ satisfying $H(x, \lambda y)=\lambda^3H(x, y)$ for all $\lambda > 0$. Such a coordinate system is called an adapted coordinate system \cite{TN1}. Recently, Shen proved that  the Finsler metric $F$ on an open subset $U\subset \mathbb{R}^n$ is dually flat if and
only if it satisfies
\[
(F^2)_{x^ky^l}y^k=2(F^2)_{x^l}.
\]
In this case, $H =-\frac{1}{6}[F^2]_{x^m}y^m$.

In this section, we will prove a generalized version of Theorem
\ref{mainthm1}. Indeed  we find necessary and sufficient condition
under which a conformal $\beta$-change of an generalized $m$-th
root metric be locally dually flat. Let $F$ be a scalar function on
$TM$ defined by $F=\sqrt{A^{2/m}+B}$,  where $A$ and $B$ are given
by \be A := a_{i_{1}}..._{i_{m}}(x)y^{i_1} . . . y^{i_m}, \ \ \    B
:=b_{ij}(x)y^iy^j.\nonumber \ee
Then $F$ is called generalized  $m$-th root Finsler metric. Suppose that the matrix $(A_{ij})$ defines a positive definite tensor and $(A^{ij})$  denotes its inverse.\\

\bigskip

Now, we  are going to prove the following:
\begin{thm}\label{mainthma}
Let  $F=\sqrt{A^{2/m}+B}$ be an generalized $m$-th root Finsler
metric on an open subset $U\subset \mathbb{R}^n$, where $A$ is
irreducible. Suppose that $\bar{F}=e^\alpha F+\beta$ be conformal
$\beta$-change of $F$ where $\beta=b_i(x)y^i,\alpha=\alpha(x)$. Then
$\bar{F}$ is locally dually flat if and only if there exists a
1-form $\theta = \theta_{l}(x)y^l$  on U such that the following
holds
\begin{eqnarray}
&&e^{2\alpha}[2B_{x^{l}}+4\alpha_{x^{l}}B-B_{0l}-2\alpha_{0}B_{l}]=2(\beta_{l}\beta_{0}+\beta\beta_{0l}-2\beta\beta_{x^{l}}),\label{SRR1}\\
&&A_{x^l}=\frac{1}{3m}[mA\theta_{l}+2\theta A_{l}+2(\alpha_{0}A_{l}-\alpha_{x^{l}}A)],\label{SRR2}\\
&&\Upsilon_l\Upsilon_0\beta=2\Upsilon[
((\Upsilon_{0}\beta)_l+\Upsilon_l\beta_0+\alpha_0\beta\Upsilon_l-2\Upsilon_{x^{l}}\beta)
+2e^\alpha\Upsilon\Psi],
\end{eqnarray}
where $\Upsilon:=A^{\frac{2}{m}}+B$, $\beta _{0l}=\beta _{x^k
y^l}y^k$ , $\alpha_0=\alpha_{x^l}y^l$,
$\beta_{x^{l}}=(b_i)_{x^{l}}y^i$, $\beta_0=(b_i)_0y^i$,
$\beta_{0l}=(b_l)_0$, and
\begin{eqnarray*}
&&\Upsilon_p=\frac{2}{m}A^{\frac{2}{m}-1}A_p+B_p,\\
&&\Upsilon_{0p}=\frac{2}{m}A^{\frac{2}{m}-2}\big[(\frac{2}{m}-1)A_pA_0+AA_{0p}\big]+B_{0p},\\
&&\Psi=\alpha_0\beta_l+\beta_{0l}-2\beta_{x^{l}}-2\alpha_{x^l}\beta.
\end{eqnarray*}
\end{thm}

\bigskip

To prove Theorem \ref{mainthma}, we need the following.
\begin{lem}\label{lemp}
Suppose that the equation $\Phi A^{\frac{2}{m}-2}+\Psi A^{\frac{1}{m}-1}+\Theta=0$ holds,  where $\Phi, \Psi,\Theta $  are polynomials in $y$ and $m>2$. Then $\Phi=\Psi=\Theta=0$.
\end{lem}

\bigskip

\noindent {\it\bf Proof of Theorem \ref{mainthma}}: The following hold
\begin{eqnarray*}
{\bar F}^2 \!\!\!\!&=&\!\!\!\!\! e^{2\alpha}(A^{\frac{2}{m}}+B)+2e^\alpha\beta (A^{\frac{2}{m}}+B)^{1/2}+\beta^2,\\
({\bar F}^2)_{x^k} \!\!\!\!&=&\!\!\!\!\! 2\alpha_{x^k}e^{2\alpha}(A^{\frac{2}{m}}+B)+e^{2\alpha}(\frac{2}{m}A^{\frac{2}{m}-1}A_{x^{k}}+B_{x^{k}})
+2\alpha_{x^k}e^\alpha\beta(A^{\frac{2}{m}}+B)^{\frac{1}{2}}
\\
&&+e^\alpha[(A^{\frac{2}{m}}+B)^{-1/2}(\frac{2}{m}A^{\frac{2}{m}-1}A_{x^k}+B_{x^{k}})\beta
+2(A^{\frac{2}{m}}+B)^{1/2}\beta_{x^{k}}]+2\beta_{x^{k}}\beta.
\end{eqnarray*}
Then
\begin{eqnarray*}
[{\bar F}^2]_{x^{k}y^{l}}y^k=\!\!\!\!&&\!\!\! 2\alpha_0
e^{2\alpha}\Upsilon_l+e^{2\alpha}\Upsilon_{0l}+2\alpha_0
e^\alpha\beta_l\Upsilon^{\frac{1}{2}}+\alpha_0
e^\alpha\beta\Upsilon^{-\frac{1}{2}}\Upsilon_l+2e^\alpha\beta_{0l}\Upsilon^{\frac{1}{2}}
\\&+&e^\alpha\beta_0\Upsilon^{-\frac{1}{2}}\Upsilon_l+e^\alpha\beta_l\Upsilon^{-\frac{1}{2}}\Upsilon_0-\frac{1}{2}e^\alpha\beta\Upsilon^{-\frac{3}{2}}\Upsilon_l\Upsilon_0+e^\alpha\beta\Upsilon^{-\frac{1}{2}}\Upsilon_{0l}
\\&+&2\beta_l\beta_0+2\beta\beta_{0l}.
\end{eqnarray*}
Since $\bar{F}$ be a locally dually flat metric, then
\begin{eqnarray*}
e^\alpha\Upsilon^{-\frac{3}{2}}\Big[\!\!\!\!&-&\!\!\!\!\!\ \frac{1}{2}\beta\Upsilon_l\Upsilon_0+\Upsilon(\beta\Upsilon_{0l}+\beta_l\Upsilon_0+
\beta_0\Upsilon_l+\alpha_0\beta\Upsilon_l-2\beta\Upsilon_{x^l})
\\ \!\!\!\!&+&\!\!\!\!\!\! 2e^\alpha\Upsilon^2(\alpha_0\beta_l+\beta_{0l}-2\alpha_{x^l}\beta-2\beta_{x^l})\Big]
\\\!\!\!\!&+&\!\!\!\!\!\! \frac{2}{m}e^{2\alpha}A^{\frac{2}{m}-2}\Big[2\alpha_0AA_l+(\frac{2}{m}-1)A_lA_0+AA_{0l}
-2\alpha_{x^l}A^2-2AA_{x^l}\Big]
\\\!\!\!\!&+&\!\!\!\!\!\! e^{2\alpha}\Big[2\alpha_0B_l+B_{0l}-4\alpha_{x^l}B-2B_{x^l}\Big]\\
\!\!\!\!&-&\!\!\!\!\!\! 4\beta\beta_{x^l}+2\beta_l\beta_0+2\beta\beta_{0l}=0.
\end{eqnarray*}
By Lemma \ref{lemp},  we have
\begin{eqnarray}
&&2\alpha_0AA_l+(\frac{2}{m}-1)A_lA_0+AA_{0l}-2\alpha_{x^l}A^2=2AA_{x^l},\label{m6}\\
&&\frac{1}{2}\beta\Upsilon_l\Upsilon_0=\Upsilon[(\beta\Upsilon_{0})_l+\beta_0\Upsilon_l+\alpha_0\beta\Upsilon_l-2\beta\Upsilon_{x^l}
+2e^\alpha\Upsilon\Psi],\\
&&e^{2\alpha}\big[2\alpha_0B_l+B_{0l}-4\alpha_{x^l}B-2B_{x^l}\big]=2(2\beta\beta_{x^l}-\beta_l\beta_0-\beta\beta_{0l}).
\end{eqnarray}
One can rewrite (\ref{m6}) as follows
\be
A(2A_{x^l}-A_{0l}+2\alpha_{x^l}A)=((\frac{2}{m}-1)A_0+2\alpha_0A)A_l.\label{d11}
\ee
Irreducibility of $A$ and $deg(A_l)=m-1$ imply that there exists
a 1-form $\theta=\theta_l y^l$ on $U$ such that
\be A_0=\theta
A.\label{d12}
\ee
By (\ref{d12}), we get
\be
A_{0l}=A\theta_l+\theta A_l-A_{x^l}.\label{d13}
\ee
Substituting (\ref{d12}) and (\ref{d13}) into (\ref{d11}) yields (\ref{SRR2}).
The converse yields by  a direct computation. This completes the proof. \qed

\section{ Proof of the Theorem \ref{mainthm2}}
It is known that
a Finsler metric  $F(x,y)$ on  ${\cal U}$  is  projective  if and only if
its geodesic coefficients $G^i$ are in the form
\[G^i(x, y) = P(x, y) y^i,\]
 where $P: T{\cal U} = {\cal U}\times \mathbb{R}^n \rightarrow \mathbb{R}$ is positively homogeneous with degree one,
$P(x,  \lambda y) = \lambda P(x, y)$, $\lambda >0$. We call $P(x, y)$ the {\it projective factor} of $F(x, y)$.
The following lemma plays an important role.

\begin{lem}\label{lemRap}{\rm (Rapcs\'{a}k)} Let $F(x, y)$ be a Finsler metric on an open subset ${\cal U} \subset \mathbb{R}^n$. $F(x, y)$ is projective on ${\cal U}$ if and only if
it satisfies
\be
F_{x^k y^l} y^k = F_{x^l}. \label{Rap}
\ee
In this case, the projective factor $P(x, y)$ is given by
\be
P = { F_{x^k} y^k\over 2 F}.\label{Gi}
\ee
\end{lem}

\bigskip
Much earlier, G. Hamel proved that a Finsler metric $F(x, y)$
on ${\cal U}\subset \mathbb{R}^n$ is projective   if and only if
\be
F_{x^ky^l} = F_{x^l y^k}.
\ee
Thus (\ref{Gi}) and (\ref{Rap}) are equivalent.

In this section, we will prove a generalized version of Theorem
\ref{mainthm2}. Indeed  we study  the conformal $\beta$-change of an
generalized $m$-th root metric $F=\sqrt{A^{\frac{2}{m}}+B}$, where
$A$ is irreducible. More precisely, we  prove the following:
\begin{thm}\label{mainthmb}
Let  $F=\sqrt{A^{2/m}+B}$ be an generalized $m$-th root Finsler
metric on an open subset $U\subset \mathbb{R}^n$, where $A$ is
irreducible. Suppose that $\bar{F}=e^\alpha F+\beta$ be conformal
$\beta$-change of $F$ where $\beta=b_i(x)y^i,\alpha=\alpha(x)$. Then
$\bar{F}$ is locally projectively flat if and only if it is locally
Minkowskian.
\end{thm}

\bigskip

To prove Theorem \ref{mainthm2},  we need the following.

\begin{lem}\label{0}
Let $(M, F)$ be a Finsler manifold. Suppose that $\bar{F}=e^\alpha
F+\beta$ be a conformal $\beta$-change of $F$. Then $\bar{F}$ is a
projectively flat Finsler metric if and only if the following holds
\be
e^\alpha(F_{0l}-F_{x^{l}})=e^\alpha(\alpha_{x^l}F-\alpha_0F_l)+(b_i)_{x^l}y^i-(b_l)_{0}.\label{Ham}
\ee
\end{lem}
\begin{proof}
The following hold
\begin{eqnarray*}
&&\bar{F}=e^\alpha F+\beta,\\
&&\bar{F}_{x^k}=\alpha_{x^k}e^\alpha F+e^\alpha F_{x^k}+(b_i)x^ky^i,\\
&&\bar{F}_0=\alpha_0e^\alpha F+e^\alpha F_0+(b_i)_0y^i,\\
&&\bar{F}_{0l}=\alpha_0e^\alpha F_l+e^\alpha F_{0}+(b_l)_0.
\end{eqnarray*}
This completes the proof.
\end{proof}
\bigskip

By using the Lemma \ref{0},  we are going to prove the following.

\begin{prop}\label{lemb}
Let  $F=\sqrt{A^{2/m}+B}$ be an generalized $m$-th root Finsler
metric on an open subset $U\subset \mathbb{R}^n$, where $A$ is
irreducible, $m>4$ and $B\neq 0$. Suppose that $\bar{F}=e^\alpha
F+\beta$ be conformal $\beta$-change of $F$ where
$\beta=b_i(x)y^i,\alpha=\alpha(x)$. In that case, if $\bar{F}$ is
projectively flat metric then $F$ reduces to a Berwald metric.
\end{prop}
\begin{proof}
By Lemma \ref{0}, we get
\[
F_{x^{l}}=\frac{2A^{2/m}A_{x^{l}}+mAB_{x^{l}}}{2mA\sqrt{A^{\frac{2}{m}}+B}},\]
Then we have
\begin{eqnarray*}
F_{x^{k}y^{l}}y^k=\!\!\!\!&&\!\!\!\! (A^{\frac{2}{m}}+B)^{-1/2}\Big[\frac{1}{4}(\frac{2A^{2/m}A_0}{mA}+B_0)
(\frac{2A^{2/m}A_l}{mA}+B_l)(A^{\frac{2}{m}}+B)^{-1}
\\ \!\!\!\!&+&\!\!\!\! \frac{1}{2}(\frac{4A^{2/m}A_0A_l}{m^2A^2}+\frac{2A^{2/m}A_{0l}}{mA}-\frac{2A^{2/m}A_0A_l}{mA^2}+B_{0l})\Big].
\end{eqnarray*}
Thus
\begin{eqnarray*}
F_{0l}-F_{x^l}\!\!\!\!&=&\!\!\!\! e^\alpha\frac{(A^{\frac{2}{m}}+B)^{-\frac{3}{2}}}{m^2A^2}\Big[A^{\frac{4}{m}}
(mAA_l\alpha_0+(1-m)A_lA_0+mAA_{0l}-mAA_{x^l})\\
&& +A^{\frac{2}{m}}(mAA_lB\alpha_0+\frac{1}{2}m^2A^2B_l\alpha_0+(2-m)A_lA_0B+mAA_{0l}B)\\
&&+\frac{1}{2}mA^{\frac{2}{m}+1}(mAB_{0l}-A_0B_l-A_lB_0-A_{x^l}B-mAB_{x^l})\\
&& +\frac{1}{2}m^2A^2(BB_l\alpha_0+B_{0l}B-\frac{1}{2}B_0B_l-BB_{x^l})\Big].
\end{eqnarray*}
By (\ref{Ham}), we obtain the following:
\[
\Phi A^{\frac{2}{m}}+\Psi A^{\frac{4}{m}}+\Theta=0,
\]
where
\begin{eqnarray*}
\Phi\!\!\!\!&=&\!\!\!\!-\frac{mA}{2}\Big[A_0B_l+B_oA_l+2B(A_{x^{l}}-A_l\alpha_0-A_{0l})+mA(B_{x^{l}}-B_l\alpha_0-B_{0l})\Big]\\
&-&(m-2)A_0A_lB,\\
\Psi\!\!\!\!&=&\!\!\!\!mA(A_{0l}+A_l\alpha_0-A_{x^{l}})-(m-1)A_0A_l,\\
\Theta\!\!\!\!&=&\!\!\!\!\frac{1}{4}m^2A^2\Big[2BB_l\alpha_0-2B_{0l}B+B_0B_l+2B_{x^{l}}B\Big],\\
\!\!\!\!&+&\!\!\!\!m^2A^2(A^{\frac{2}{m}}+B)^\frac{3}{2}e^{-2\alpha}\Big[(b_l)_{0}-(b_i)_{x^{l}}y^i
+e^\alpha(\alpha_{x^l}A^{\frac{1}{m}}-\frac{1}{m}\alpha_0A^{\frac{1}{m}-1}A_l)\Big].
\end{eqnarray*}
By Lemma \ref{lemp},  we have
\begin{eqnarray}
&&\Phi=0,\label{n}\\
&&\Psi=0,\label{m}\\
&&\Theta=0.
\end{eqnarray}
By (\ref{m}),  it follows that
\be
mA(A_l\alpha_0+A_{0l}-A_{x^{l}})=(m-1)A_0A_l.\label{w}
\ee
Then
irreducibility of $A$ and $deg(A_l )= m-1<deg(A)$ implies that $A_0$
is divisible by $A$. This means that, there is a 1-form
$\theta=\theta_l y^l$ on $U$ such that, \be\label{q} A_0=2m A\theta.
\ee Substituting (\ref{q}) into (\ref{w}), yields \be\label{e}
A_{0l}=A_{x^{l}}-A_l\alpha_0+2(m-1)\theta A_l. \ee Plugging
(\ref{q}) and (\ref{e}) into (\ref{n}), we get \be mA(2\theta
B_l-B_{0l}-B_l\alpha_0+B_{x^{l}})=A_l(4B\theta-B_0).\label{r} \ee
Clearly, the right side of (\ref{r}) is divisible by $A$. Since $A$
is irreducible, deg($A_l$)  and
 deg($ 2\theta B-\frac{1}{2}B$) are both less than deg$(A)$, then we have
\be\label{17}
B_0=4 B\theta.
\ee
By (\ref{q}) and (\ref{17}), we get the spray coefficients
$G^i = Py^i $ with $P=\theta$. Thus $F$ is a Berwald metric.
\end{proof}

\bigskip

The Riemann curvature ${\bf K}_y= R^i_{\ k}  dx^k \otimes \pxi|_x : T_xM \to T_xM$ is a family of linear maps on tangent spaces,
defined by
\[
R^i_{\ k} = 2 {\pa G^i\over \pa x^k}-y^j{\pa^2 G^i\over \pa x^j\pa y^k}
+2G^j {\pa^2 G^i \over \pa y^j \pa y^k} - {\pa G^i \over \pa y^j}{\pa G^j \over \pa y^k}.  \label{Riemann}
\]
For a flag $P={\rm span}\{y, u\} \subset T_xM$ with flagpole $y$, the flag curvature ${\bf K}={\bf K}(P, y)$ is defined by
\[
{\bf K}(P, y):= {\g_y (u, {\bf K}_y(u))
\over \g_y(y, y) \g_y(u,u)
-\g_y(y, u)^2 },
\]
When $F$ is Riemannian, ${\bf K}={\bf  K}(P)$ is independent of $y\in P$, which  is just the sectional curvature of $P$ in Riemannian geometry. We say that  a Finsler metric $F$ is   of scalar curvature if for any $y\in T_xM$, the flag curvature ${\bf K}= {\bf K}(x, y)$ is a scalar function on the slit tangent bundle $TM_0$.  One of the important problems in Finsler geometry is to characterize Finsler manifolds of scalar flag curvature \cite{NST}\cite{NT}. If ${\bf K}=constant $, then the Finsler metric $F$ is said to be of constant flag  curvature.

\bigskip

\noindent {\it\bf Proof of Theorem \ref{mainthmb}}: By Proposition \ref{lemb},   $F$ is  a Berwald metric. On the other hand, according to  Numata's Theorem every Berwald metric of non-zero scalar flag curvature ${\bf K}$ must be Riemaniann. This is contradicts with our assumption. Then ${\bf K}=0$, and in this case $F$ reduces to a locally Minkowskian metric.
\qed

\bigskip
\noindent
Akbar Tayebi and Ali Nankali\\
Department of Mathematics, Faculty  of Science\\
University of Qom \\
Qom. Iran\\
Email:\ akbar.tayebi@gmail.com\\
Email:\ ali.nankali2327@yahoo.com
\bigskip

\noindent
Esmaeil Peyghan\\
Department of Mathematics, Faculty  of Science\\
Arak University\\
Arak 38156-8-8349.  Iran\\
Email: epeyghan@gmail.com

\end{document}